\documentclass[reqno]{amsart}
\usepackage[top=2cm,bottom=2cm,right=2.5cm,left=2.5cm]{geometry}
\usepackage{amssymb}
\usepackage{amsmath, amsthm, amscd, amsfonts, amssymb, graphicx, color}
\usepackage[bookmarksnumbered, colorlinks, plainpages]{hyperref}
\hypersetup{colorlinks=true,linkcolor=red, anchorcolor=green, citecolor=cyan, urlcolor=red, filecolor=magenta, pdftoolbar=true}

\usepackage{hyperref}

\newtheorem{theorem}{Theorem}[section]
\newtheorem{lemma}[theorem]{Lemma}

\newtheorem{corollary}[theorem]{Corollary}
\theoremstyle{definition}
\newtheorem{definition}[theorem]{Definition}
\newtheorem{example}[theorem]{Example}

\theoremstyle{remark}

\numberwithin{equation}{section}

\begin{document}

\setcounter{page}{1}

\title[Fixed point theorems for generalized
$\theta-\phi-$contraction mappings]{Fixed point theorems for generalized
$\theta-\phi-$contraction mappings in rectangular
quasi b-metric spaces}
\author[M. Rossafi, A. Kari]{Mohamed Rossafi$^{1*}$ and Abdelkarim Kari$^{2}$}

\address{$^{1}$Higher School of Education and Training, University of Ibn Tofail, Kenitra, Morocco}
\email{\textcolor[rgb]{0.00,0.00,0.84}{rossafimohamed@gmail.com}}	
	
	\address{$^{2}$ Laboratory of Analysis, Modeling and Simulation Faculty of Sciences Ben M’Sik, Hassan II University, B.P. 7955 Casablanca, Morocco}
	\email{\textcolor[rgb]{0.00,0.00,0.84}{abdkrimkariprofes@gmail.com}}

\date{
\newline \indent $^{*}$Corresponding author}


\begin{abstract}
A generalized version of both rectangular metric spaces and rectangular quasi-metric spaces is known as rectangular quasi b-metric spaces (RQB-MS).
In the current work, we define generalized $( \theta,\phi) $-contraction mappings and study fixed point (FP) results for the maps introduced in the setting of rectangular quasi b-metric spaces. Our results generalize many existing results. We also provide examples in support of our main findings.
\end{abstract}

\maketitle
\section{Introduction }
The Banach  contraction principle is a basic result in fixed point theory (FPT) \cite{BA}. Due to its importance, various mathematicians studied many interesting extensions and generalizations, (see \cite{BRO,KAN,RE}).\\
Numerous generalizations of the concept of metric spaces (MS) are defined and some FPTs have been proved in these spaces. For instance, asymmetric MS were introduced by Wilson \cite{W} as a generalization of MS. Many mathematicians worked on this interesting space ( see also \cite{AZ}). Branciari (\cite{BRA}) seems to be the first to generalize MS in 2000.  In the generalization, the triangle inequality is replaced by the quadrilateral inequality  $ d(x,y)\leq d(x,z)+d(z,u)+d(u,y) $ for all pairwise distinct points $ x,y,z $ and $ u $. Any MS is a generalized MS but in general, generalized MS might not be a MS. Various FP results were established on such spaces, (see \cite{JS,KARIKO,KIR} for more details).
\par
The notion of b-rectangular MS have been introduced by the authors in \cite{RG}, and many authors investigated many existing FPTs in such spaces, (see e.g. \cite{KARO,RO}). Bontu Nasir et al in \cite{PIR} introduced the notions of quasi b-generalized MS. Any generalized MS but in general, RQB-MS might not be a generalized MS. The concept of $\theta-\phi-$contraction has been introduced by Zheng et al. in \cite{ZH}. They also established some FP results for such mappings in complete MS and generalized the results of Kannan and Brower.
\par
In the current paper, we introduce a new notion of generalized $ \theta-\phi-$contraction and establish some results of FP for such mappings in complete rectangular quasi b-metric. The results presented in the paper extend the corresponding results of Zheng et al. \cite{ZH} and Banach \cite{KAN} on RQB-MS. Also, we derive some useful corollaries of these results.

\section{preliminaries}

In this section, we give basic notions concerning a $\theta-\phi-$contraction in the setting of b-MS.
\begin{definition}
\cite{PIR} Suppose a non-empty set $\mathcal{X}$ and $ \rho:\mathcal{X}\times \mathcal{X}\rightarrow \mathbb{R}^{+}$ be a mapping such that $\forall x_1,x_2$ $\in \mathcal{X}$ and $\forall$ distinct points $u,v\in \mathcal{X}$, each of them different from $x_1$ and $x_2,$ on has
\item[(i)] $\rho(x_1,x_2)=0$ if and only if $x_1=x_2;$
\item[(ii)] $\rho(x_1,x_2)\leq s\left[ \rho(x_1,u)+\rho(u,v)+\rho(v,x_2)\right] .\left( \text{b-rectangular inequality}\right)$\\
Then $\left(\mathcal{X},\rho\right) $ is called an QRB-MS.
\end{definition}

\begin{definition}
\cite{PIR}. Suppose a QRB-MS $\left(\mathcal{X},\rho\right) $ and $\left\lbrace x_{n}\right\rbrace _{n\in\mathbb{N}}$ be a sequence in $\mathcal{X}$, and $x\in \mathcal{X}$. Then\\
\item[(i)] The sequence $\left\lbrace x_{n}\right\rbrace _{n\in\mathbb{N}}$ forward (backward) converges to $x$ if and only if
$$\lim\limits_{n\rightarrow +\infty}\rho \left( x,x_{n}\right)=\lim\limits_{n\rightarrow +\infty}\rho \left( x_{n},x\right) =0. $$
\item[(ii)] The sequence $\left\lbrace x_{n}\right\rbrace _{n\in\mathbb{N}}$ forward (backward) Cauchy  if
$$\lim\limits_{n,m\rightarrow +\infty}\rho\left( x_{n}, x_{m}\right) =\lim\limits_{n,m\rightarrow +\infty}\rho \left( x_{m},x_{n}\right)=0. $$
\end{definition}

 \begin{example}
Define $ \mathcal{X}:=A\cup B $, where $ A=\lbrace \frac{1}{n}:n\in \mathbb{N}, 2\leq n\leq 7 \rbrace $ and $ B=\left[1,2 \right]  $. Define $ \rho:\mathcal{X}\times \mathcal{X}\rightarrow \left[0,+\infty \right[  $ as follows:
	\begin{equation*}
	\left\lbrace
	\begin{aligned}
	\rho(a, b) &=\rho(b, a)\ for \ all \  a,b\in \mathcal{X}.\\
	\end{aligned}
	\right.
	\end{equation*}
	and
	\begin{equation*}
	\left\lbrace
	\begin{aligned}		
	\rho\left( \frac{1}{2},\frac{1}{3}\right) =\rho\left( \frac{1}{4},\frac{1}{5}\right) =\rho\left( \frac{1}{6},\frac{1}{7}\right) 	&=0,05\\
	\rho\left( \frac{1}{3},\frac{1}{2}\right) =\rho\left( \frac{1}{5},\frac{1}{4}\right) =\rho\left( \frac{1}{7},\frac{1}{6}\right) 	&=0,04\\
	\rho\left( \frac{1}{2},\frac{1}{4}\right) =\rho\left( \frac{1}{3},\frac{1}{7}\right) =\rho\left( \frac{1}{5},\frac{1}{6}\right) 	&=0,08\\
	\rho\left( \frac{1}{4},\frac{1}{2}\right) =\rho\left( \frac{1}{7},\frac{1}{3}\right) =\rho\left( \frac{1}{6},\frac{1}{5}\right) 	&=0,05\\
	\rho\left( \frac{1}{2},\frac{1}{6}\right) =\rho\left( \frac{1}{3},\frac{1}{4}\right) =\rho\left( \frac{1}{5},\frac{1}{7}\right) 	&=0,4\\
	\rho\left( \frac{1}{2},\frac{1}{5}\right) =\rho\left( \frac{1}{3},\frac{1}{6}\right) =\rho\left( \frac{1}{4},\frac{1}{7}\right) 	&=0,24\\
	\rho\left( \frac{1}{2},\frac{1}{7}\right) =\rho\left( \frac{1}{3},\frac{1}{5}\right) =\rho\left( \frac{1}{4},\frac{1}{6}\right) 	&=0,15\\
	\rho\left(a,b\right) =\left( \vert a-b\vert\right) ^{2} \ otherwise.
	\end{aligned}
	\right.
	\end{equation*}
	Then $ (\mathcal{X},\rho) $ is a QRB-MS with coefficient $s=3$.
\end{example}

The following notion was introduced in \cite{JS}.
\begin{definition}
 \cite{JS}. Suppose $ \Theta $ be the family of all increasing and continuous functions
 $\theta : \left]0,+\infty \right[ \rightarrow \left] 1 ,+\infty \right[$: For each sequence $(x_{n})\subset \left] 0,+\infty \right[$;
\begin{equation*}
\lim_{n\rightarrow 0}x_{n}=0\ \,\,\text{ if and only if}\,\,\,\lim_{n\rightarrow \infty }\theta\left( x_{n}\right) =1;\\
\end{equation*}
\end{definition}

In \cite{ZH}. Zheng et al. presented the concept of $ \theta-\phi-$contraction on MS.
\begin{definition}
\cite{ZH} Let $ \Phi $ be the family of all nondecreasing and continuous functions $\phi$:  $\left[ 1,+\infty \right[ $ $\rightarrow \left[ 1,+\infty \right[ $: For each  $t\in \left] 1,+\infty \right[ $, $  lim_{n\rightarrow \infty }\phi^{n}( t) =1$.
\end{definition}

It should be remarked also that the authors in \cite{ZH} proved the following nice results.
\begin{lemma}\label{2.7}
\cite{ZH} If $\phi $ $\in \Phi$. Then $\phi(t)< t $ for all $t\in \left]1, \infty\right[$ and $ \phi(1)$=1.
\end{lemma}

\begin{definition}
\cite{ZH}. For a MS $(\mathcal{X},\rho)$ and a mapping $\mathcal{T}:\mathcal{X}\rightarrow \mathcal{X}$.\\
 $\mathcal{T}$ is called a $\theta-\phi-$contraction if there exist $\theta \in \Theta $ and $\phi \in \Phi $: for any $a,b\in \mathcal{X},$
\begin{equation*}
\rho\left( \mathcal{T}a,\mathcal{T}b\right) >0\Rightarrow \theta \left[ \rho\left( \mathcal{T}a,\mathcal{T}b\right) \right]\leq \phi \left( \theta \left[ N\left( a,b\right) \right] \right),
\end{equation*}
where
\begin{equation*}
N\left(a,b\right) =\max \left\{ \rho\left( a,b\right) ,\rho\left( a,\mathcal{T}a\right),\rho\left( b,\mathcal{T}b\right) \right\}.
\end{equation*}
\end{definition}

\begin{theorem}
\cite{ZH}. For a complete MS $\left( \mathcal{X},d\right)$ and a $ \theta-\phi-$contraction $\mathcal{T}:\mathcal{X}\rightarrow \mathcal{X}$. Then $\mathcal{X}$ has a unique FP.
\end{theorem}

\section{Main result}

\begin{lemma}\label{2.3}
Suppose a QRB-MS $\left(\mathcal{X}, \eta\right)$. Let sequences $\lbrace  x_{n}\rbrace $ and $\lbrace y_{n}\rbrace $ in $\mathcal{X}$. Then
\begin{itemize}
 \item[(a)] If $x_{n} $ forward convergent to $  x$ and $y_{n}$ backward convergent to $ y $ as $n\rightarrow \infty ,$ with $x\neq y,$ $x_{n}\neq x$ and $y_{n}\neq y$ for all $n\in \mathbb{N}.$ Then we have
\begin{equation*}
\frac{1}{s}\eta\left( x,y\right) \leq \lim_{n\rightarrow \infty }\inf \eta\left(x_{n},y_{n}\right) \leq \lim_{n\rightarrow \infty }\sup \eta\left(x_{n},y_{n}\right).
\end{equation*}
\item[(b)] If $x_{n} $ backward convergent to $  x$ and $y_{n}$ forward convergent to $ y $ as $n\rightarrow \infty ,$ with $x\neq y,$ $x_{n}\neq x$ and $y_{n}\neq y$ for all $n\in \mathbb{N}.$ Then we have
\begin{equation*}
\frac{1}{s}\eta\left( y,x\right) \leq \lim_{n\rightarrow \infty }\inf \eta\left(y_{n},x_{n}\right) \leq \lim_{n\rightarrow \infty }\sup \eta\left(y_{n},x_{n}\right).
\end{equation*}
\item[(c)]  If $y\in \mathcal{X}$ and $\lbrace x_{n}\rbrace $ is a Cauchy sequence (CS) in $\mathcal{X}$ with $x_{n}\neq x_{m}$ for any $m,n\in \mathbb{N},$ $m\neq n,$ converging to $x\neq y,$ then
\begin{equation*}
\frac{1}{s}\eta\left( x,y\right) \leq \lim_{n\rightarrow \infty }\inf \eta\left(x_{n},y\right) \leq \lim_{n\rightarrow \infty }\sup \eta\left( x_{n},y\right).
\end{equation*}
for all $x\in \mathcal{X}.$
\item[(d)]  If $y\in \mathcal{X}$ and $\lbrace x_{n}\rbrace $ is a CS in $\mathcal{X}$ with $x_{n}\neq x_{m}$ for any $m,n\in \mathbb{N},$ $m\neq n,$ converging to $x\neq y,$ then
\begin{equation*}
\frac{1}{s}\eta\left( x,y\right) \leq \lim_{n\rightarrow \infty }\inf \eta\left(y,x_{n}\right) \leq \lim_{n\rightarrow \infty }\sup \eta\left( y,x_{n}\right).
\end{equation*}
for all $x\in \mathcal{X}.$
\end{itemize}
\end{lemma}

\begin{proof}
Using the b-rectangular inequality, it is easy to see that
$$  \eta(x, y) \leq  s\left[ \eta(x, x_{n}) + s\eta(x_{n}; y_{n}) + s\eta(y_{n}; y)\right] $$
$$  \eta(y, x) \leq  s\left[ \eta(y, y_{n}) + s\eta(y_{n}; x_{n}) + s\eta(x_{n}; x)\right],$$
taking the lower limit as $n\rightarrow\infty  $ and the upper limit as $n\rightarrow\infty $, we obtain
 \begin{itemize}
\item[(a)]
\begin{equation*}
\frac{1}{s}\eta\left( x,y\right) \leq \lim_{n\rightarrow \infty }\inf \eta\left(x_{n},y_{n}\right) \leq \lim_{n\rightarrow \infty }\sup \eta\left(x_{n},y_{n}\right),
\end{equation*}
and
\item[(b)]
\begin{equation*}
\frac{1}{s}\eta\left( y,x\right) \leq \lim_{n\rightarrow \infty }\inf \eta\left(y_{n},x_{n}\right) \leq \lim_{n\rightarrow \infty }\sup \eta\left(y_{n},x_{n}\right).
\end{equation*}
If $  y\in \mathcal{X}$, then, for infinitely many $ m, n \in\mathbb{N}$,
$$ \eta(x, y) \leq s\eta(x, x_{n}) + s\eta(x_{n},x_{m}) + s\eta(x_{m}, y), $$
and
$$ \eta(y, x) \leq  s\eta(x_{n},x ) + s\eta(x_{m},x_{n}) + s\eta(x_{n}, y). $$
Taking the lower limit as $n\rightarrow\infty  $ and the upper limit as $n\rightarrow\infty$, we obtain
\item[(c)]
\begin{equation*}
\frac{1}{s}\eta\left( x,y\right) \leq \lim_{n\rightarrow \infty }\inf \eta\left(x_{n},y\right) \leq \lim_{n\rightarrow \infty }\sup \eta\left( x_{n},y\right),
\end{equation*}
and
\item[(d)]
\begin{equation*}
\frac{1}{s}\eta\left( x,y\right) \leq \lim_{n\rightarrow \infty }\inf \eta\left(y,x_{n}\right) \leq \lim_{n\rightarrow \infty }\sup \eta\left( y,x_{n}\right).
\end{equation*}
\end{itemize}
\end{proof}

\begin{lemma}\label{lemmkari}
Suppose an QRB-MS $\left( \mathcal{X},\eta\right)$  and $\left\lbrace x_{n}\right\rbrace _{n}$ be a forward (or backward) CS with pairwise disjoint elements in $\mathcal{X}$. If $\left\lbrace x_{n}\right\rbrace _{n}$ forward converges to $ x \in \mathcal{X} $ and backward converges to $ y \in \mathcal{X} $, then $x=y$.
\end{lemma}
\begin{proof}
Fix $ \varepsilon > 0 $. First assume that $\left\lbrace x_{n}\right\rbrace _{n}$ is a forward CS, so there
exists $   n_{1}\in\mathbb{N} $: $  \eta(x_{n}, x_{m}) < \frac{\varepsilon}{3s} $ for all $   m \geq n \geq  n_{1}$. Since $ x_{n} $ forward converges to $ x  $ so there exists $ n_{2}\in\mathbb{N} $: $ \eta(x,x_{n}\leq \frac{\varepsilon}{3s}$. Also $ x_{n} $ backward converges to $ y $ there exists $ n_{3}\in\mathbb{N} $: $  \eta(x_{n}, x_{m}) < \frac{\varepsilon}{3s} $ for all $ m \geq n \geq  n_{3}$. Then for all $ l\geq \max\lbrace  n_{1}, n_{2}, n_{3}\rbrace $,
$  \eta(x, y) \leq s\left[ \eta(x, x_{n})+\eta(x_{n},x_{m})+\eta(x_{m},y)\right] <s\frac{\varepsilon}{3s}+s\frac{\varepsilon}{3s}+s\frac{\varepsilon}{3s}=\varepsilon $. As $  \varepsilon > 0 $ was arbitrary, we deduce that $  \eta(x, y)=0$, which implies $ x=y$. When $\left\lbrace x_{n}\right\rbrace _{n}$ is a backward CS, the proof is in similar fashion.
\end{proof}

\begin{lemma}\label{4.4}
Let $\left(\mathcal{X}, \eta\right) $ be a RQB-MS and let $\lbrace x_n \rbrace$ be a sequence in $ \mathcal{X} $:
 \begin{equation}
 \lim_{n\rightarrow \infty } \eta\left(x_{n},x_{n+1}\right)= \lim_{n\rightarrow \infty }\ \eta\left(x_{n},x_{n+2}\right)=0,
 \end{equation}
 and
 \begin{equation}
 \lim_{n\rightarrow \infty } \eta\left(x_{n+1},x_{n}\right)= \lim_{n\rightarrow \infty }\ \eta\left(x_{n+2},x_{n}\right)=0.
 \end{equation}
 If $ \lbrace x_n \rbrace $ is not a CS, then there exist $ \varepsilon >0 $ and two sequences $ \lbrace m_{\left( k\right)} \rbrace $ and $ \lbrace n_{\left( k\right)} \rbrace $ of positive integers:
 $$\varepsilon \leq \lim_{k\rightarrow \infty }\inf \eta\left( x_{m_{\left( k\right) }},x_{n_{\left( k\right)}}\right)  \leq \lim_{k\rightarrow \infty }\sup \eta\left( x_{m_{\left( k\right) }},x_{n_{\left( k\right)}}\right)\leq s\varepsilon ,$$
  $$\varepsilon \leq \lim_{k\rightarrow \infty }\inf \eta\left( x_{n_{\left( k\right) }},x_{m_{\left( k\right)+1}}\right)  \leq \lim_{k\rightarrow \infty }\sup \eta\left( x_{n_{\left( k\right) }},x_{m_{\left( k\right)+1}}\right)\leq s\varepsilon ,$$
  $$\varepsilon \leq \lim_{k\rightarrow \infty }\inf \eta\left( x_{m_{\left( k\right) }},x_{n_{\left( k\right)+1}}\right)  \leq \lim_{k\rightarrow \infty }\sup \eta\left( x_{m_{\left( k\right) }},x_{n_{\left( k\right)+1}}\right)\leq s\varepsilon ,$$
  $$\frac{\varepsilon}{s} \leq \lim_{k\rightarrow \infty }\inf \eta\left( x_{m_{\left( k\right)+1 }},x_{n_{\left( k\right)+1}}\right)  \leq \lim_{k\rightarrow \infty }\sup \eta\left( x_{m_{\left( k\right)+1 }},x_{n_{\left( k\right)+1}}\right)\leq s^{2}\varepsilon $$
  $$\varepsilon \leq \lim_{k\rightarrow \infty }\inf \eta\left( x_{n_{\left( k\right) }},x_{m_{\left( k\right)}}\right)  \leq \lim_{k\rightarrow \infty }\sup \eta\left( x_{n_{\left( k\right) }},x_{m_{\left( k\right)}}\right)\leq s\varepsilon ,$$
  $$\varepsilon \leq \lim_{k\rightarrow \infty }\inf \eta\left( x_{m_{\left( k\right) }},x_{n_{\left( k\right)+1}}\right)  \leq \lim_{k\rightarrow \infty }\sup \eta\left( x_{m_{\left( k\right) }},x_{n_{\left( k\right)+1}}\right)\leq s\varepsilon ,$$
  $$\varepsilon \leq \lim_{k\rightarrow \infty }\inf \eta\left( x_{n_{\left( k\right) }},x_{m_{\left( k\right)+1}}\right)  \leq \lim_{k\rightarrow \infty }\sup \eta\left( x_{n_{\left( k\right) }},x_{m_{\left( k\right)+1}}\right)\leq s\varepsilon ,$$
  $$\frac{\varepsilon}{s} \leq \lim_{k\rightarrow \infty }\inf \eta\left( x_{n_{\left( k\right)+1 }},x_{m_{\left( k\right)+1}}\right)  \leq \lim_{k\rightarrow \infty }\sup \eta\left( x_{n_{\left( k\right)+1 }},x_{m_{\left( k\right)+1}}\right)\leq s^{2}\varepsilon .$$
\end{lemma}

\begin{proof}
If $ \lbrace x_n \rbrace $ s not a CS, then there exist $ \varepsilon >0 $ and two sequences $ \lbrace m_{\left( k\right)} \rbrace $ and $ \lbrace n_{\left( k\right)} \rbrace $ of positive integers:
\begin{equation}
 m(k)>n(k)>k, \ \varepsilon \leq \eta\left( x_{m_{\left( k\right) }},x_{n_{\left( k\right)}}\right) \ , \ \eta\left( x_{m_{\left( k\right)-1 }},x_{n_{\left( k\right)}}\right)< \varepsilon,
 \end{equation}
and
 \begin{equation}\label{f}
  \varepsilon \leq \eta\left( x_{n_{\left( k\right) }},x_{m_{\left( k\right)}}\right) \ , \ \eta\left( x_{n_{\left( k\right)-1 }},x_{m_{\left( k\right)}}\right)< \varepsilon,
 \end{equation}
for all positive integers $k$. By the b-rectangular inequality, we have
\begin{equation}\label{ss}
 \varepsilon \leq \eta\left( x_{m_{\left( k\right) }},x_{n_{\left( k\right)}}\right)\leq s\left[ \eta\left( x_{m_{\left( k\right) }},x_{m_{\left( k\right)+1}}\right)+\eta\left( x_{m_{\left( k\right) +1}},x_{m_{\left( k\right)-1}}\right)+\eta\left( x_{m_{\left( k\right)-1 }},x_{n_{\left( k\right)}}\right)\right].
\end{equation}
Taking the upper and lower limits as $ k\rightarrow \infty $ in $ (\ref{ss}) $, we obtain
 \begin{equation}
 \varepsilon \leq \lim_{k\rightarrow \infty }\inf \eta\left( x_{m_{\left( k\right) }},x_{n_{\left( k\right)}}\right)  \leq \lim_{k\rightarrow \infty }\sup \eta\left( x_{m_{\left( k\right) }},x_{n_{\left( k\right)}}\right)\leq s\varepsilon.
 \end{equation}
Using the b-rectangular inequality again, we have
  \begin{equation}\label{dd}
 \varepsilon \leq \eta\left( x_{n_{\left( k\right) }},x_{m_{\left( k\right)+1}}\right)\leq s\left[ \eta\left( x_{n_{\left( k\right) }},x_{m_{\left( k\right)-1}}\right)+\eta\left( x_{m_{\left( k\right) -1}},x_{m_{\left( k\right)}}\right)+\eta\left( x_{m_{\left( k\right) }},x_{m_{\left( k\right)+1}}\right)\right].
\end{equation}
 Taking the upper and lower limits as $ k\rightarrow \infty $ in $ (\ref{dd}) $, we obtain
  \begin{equation}
 \varepsilon \leq \lim_{k\rightarrow \infty }\inf \eta\left( x_{n_{\left( k\right) }},x_{m_{\left( k\right)+1}}\right)  \leq \lim_{k\rightarrow \infty }\sup \eta\left( x_{n_{\left( k\right) }},x_{m_{\left( k\right)+1}}\right)\leq s\varepsilon.
 \end{equation}
Going the same way, we can easily prove the rest of the inequalities.
\end{proof}

\begin{definition}
Let $(\mathcal{X},\eta)$ be a QRB-MS with parameter $s>1$ space and $\mathcal{T}:\mathcal{X}\rightarrow \mathcal{X}$ be a mapping. $\mathcal{T}$ is called a $\theta-$contraction if there exist $\theta \in \Theta $ and $r\in \left] 0,1\right[ $:
\begin{equation*}
\eta\left( \mathcal{T}x,\mathcal{T}y\right) >0\Rightarrow \theta \left[s^{2}\eta\left( \mathcal{T}x,\mathcal{T}y\right)\right]  \leq  \theta \left[ \eta\left( x,y\right) \right] ^{r}.
\end{equation*}
\end{definition}

\begin{theorem}\label{Theorem3.5}
Let $\left( \mathcal{X},\eta\right) $ be a complete QRB-MS and let $\mathcal{T}:\mathcal{X}\rightarrow \mathcal{X}$ be an $ \theta $-contraction, i.e, there exist $\theta \in \Theta $ and $r\in \left] 0,1\right[ $: for any $ x,y \in \mathcal{X} $, we have
\begin{equation}\label{3.1}
\eta\left( \mathcal{T}x,\mathcal{T}y\right) >0\Rightarrow \theta \left[s^{2}\eta\left( \mathcal{T}x,\mathcal{T}y\right)\right]  \leq  \theta \left[ \eta\left( x,y\right) \right] ^{r}.
\end{equation}
Then $\mathcal{T}$ has a unique FP.
\end{theorem}

\begin{proof}
Let $x_{0}\in \mathcal{X}$ be an arbitrary point in $ \mathcal{X} $ and define a sequence $\left\lbrace x_{n}\right\rbrace$ by
$$x_{n+1} =\mathcal{T}x_{n}=\mathcal{T}^{n+1}x_{0},$$
for all $n\in \mathbb{N}.$ If there exists $n_0\in \mathbb{N}$ such that $\eta\left( x_{n_0},x_{n_0+1}\right) =0$, then proof is finished.\\
We can suppose that $\eta\left( x_{n},x_{n+1}\right) >0$ for all $n\in \mathbb{N}$.
Substituting $x=x_{n-1}$ and $y=x_{n}$, from $(\ref{3.1})$, for all $n\in  $ $\mathbb{N}$, we have
\begin{equation}\label{3.2}
\theta \left[ \eta\left( x_{n},x_{n+1}\right)\right]\leq \theta \left[ s^{3}\eta\left( x_{n},x_{n+1}\right)\right] \leq \left[ \theta \left( \eta\left(x_{n-1},x_{n}\right) \right) \right] ^{r},\forall n\in \mathbb{N}
\end{equation}
Repeating this step, we conclude that
\begin{equation*}\theta \left( \eta\left( x_{n},x_{n+1}\right) \right)  \leq \left(\theta \left( \eta\left( x_{n-1},x_{n}\right) \right) \right) ^{r} \\
\leq \left(\theta \left( \eta\left( x_{n-2},x_{n-1}\right) \right) \right) ^{r^{2}}\leq ...\leq\theta\left( \eta\left( x_{0},x_{1}\right)\right) ^{r^{n}}.
 \end{equation*}
From $(3.3) $ and using $\left(\theta _{1}\right) $ we get
\begin{equation}
\eta\left( x_{n},x_{n+1}\right) <\eta\left( x_{n-1},x_{n}\right).
\end{equation}
Therefore, $\eta\left( x_{n,}x_{n+1}\right) _{n\in \mathbb{N}}$ is monotone strictly decreasing sequence of non negative real numbers. Consequently, there exists $\alpha \geq 0$:
\begin{equation*}
\lim_{n\rightarrow \infty }\eta\left( x_{n+1,}x_{n}\right) \ \ =\alpha.
\end{equation*}
Now, we claim that $\alpha=0$. Arguing by contradiction, we assume that $\alpha >0.$ Since $\eta\left( x_{n,}x_{n+1}\right) _{n\in \mathbb{N}}$ is a non negative decreasing sequence, then we have
\begin{equation*}
\eta\left( x_{n,}x_{n+1}\right)  \geq \alpha \text{ \ \ }\forall n\in \mathbb{N}.\\
\end{equation*}
By property of $ \theta $ we get,
\begin{equation}\label{co}
1<\theta \left( \alpha\right) \leq \theta \left(\eta\left(x_{0},x_{1}\right)\right) ^{r^{n}}.
\end{equation}
By letting $n\rightarrow \infty $ in inequality $\left( \ref{co}\right),$ we obtain
\begin{equation*}
1<\theta \left(\alpha \right) \leq 1.
\end{equation*}
It is a contradiction. Therefore,
\begin{equation}\label{3.6}
\lim_{n\rightarrow \infty }\eta\left( x_{n,}x_{n+1}\right) =0.
\end{equation}
Substituting $x=x_{n-1}$ and $y=x_{n+1}$, from $(\ref{3.1})$, for all $n\in  $ $\mathbb{N}$, we have
\begin{equation}\label{3.20}
\theta \left[ \eta\left( x_{n},x_{n+2}\right)\right]\leq \theta \left[ s^{2}\eta\left( x_{n},x_{n+2}\right)\right] \leq \left[ \theta \left( \eta\left(x_{n-1},x_{n}\right) \right) \right] ^{r},\forall n\in \mathbb{N}.
\end{equation}

Repeating this step, we conclude that
\begin{equation*}\theta \left( \eta\left( x_{n},x_{n+2}\right) \right)  \leq \left(\theta \left( \eta\left( x_{n-1},x_{n+1}\right) \right) \right) ^{r} \\
\leq \left(\theta \left( \eta\left( x_{n-2},x_{n}\right) \right) \right) ^{r^{2}}\leq ...\leq\theta\left( \eta\left( x_{0},x_{2}\right)\right) ^{r^{n}}.
 \end{equation*}
From $(\ref{3.20}) $ and using $\left( \theta _{1}\right) $ we get
\begin{equation}
\eta\left( x_{n},x_{n+2}\right) <\eta\left( x_{n-1},x_{n+1}\right).
\end{equation}

Therefore, $\eta\left( x_{n,}x_{n+2}\right) _{n\in \mathbb{N}}$ is monotone strictly decreasing sequence of non negative real\\ numbers. Consequently, there exists $\delta \geq 0$:
\begin{equation*}
\lim_{n\rightarrow \infty }\eta\left( x_{n+1,}x_{n}\right) \ \ =\delta.\\
\end{equation*}
Now, we claim that $\alpha =0$. Arguing by contradiction, we assume that $\delta >0.$ Since $\eta\left( x_{n,}x_{n+2}\right) _{n\in \mathbb{N}}$ is a non negative decreasing sequence, then we have
\begin{equation*}
\eta\left( x_{n,}x_{n+2}\right)  \geq \delta \text{ \ \ }\forall n\in \mathbb{N}.
\end{equation*}

By property of $ \theta $ we get,
\begin{equation}
1<\theta \left( \delta\right) \leq \theta \left(\eta\left(x_{0},x_{2}\right)\right) ^{r^{n}}.
\end{equation}
By letting $n\rightarrow \infty $ in inequality $\left( 3.5\right),$ we obtain
\begin{equation*}
1<\theta \left(\delta \right) \leq 1.
\end{equation*}
It is a contradiction. Therefore,
\begin{equation}\label{3.6}
\lim_{n\rightarrow \infty }\eta\left( x_{n,}x_{n+2}\right) =0.
\end{equation}
Substituting $x=x_{n}$ and $y=x_{n-1}$, from $(\ref{3.1})$, for all $n\in  $ $\mathbb{N}$, we have
\begin{equation}\label{3.21}
\theta \left[ \eta\left( x_{n+1},x_{n}\right)\right]\leq \theta \left[ s^{2}\eta\left( x_{n+1},x_{n}\right)\right] \leq \left[ \theta \left( \eta\left(x_{n},x_{n-1}\right) \right) \right] ^{r},\forall n\in \mathbb{N}.
\end{equation}

Repeating this step, we conclude that
\begin{equation*}\theta \left( \eta\left( x_{n},x_{n+1}\right) \right)  \leq \left(\theta \left( \eta\left( x_{n-1},x_{n}\right) \right) \right) ^{r} \\
\leq \left(\theta \left( \eta\left( x_{n-2},x_{n-1}\right) \right) \right) ^{r^{2}}\leq ...\leq\theta\left( \eta\left( x_{0},x_{1}\right)\right) ^{r^{n}}.
 \end{equation*}
From $(\ref{3.21}) $ and using $\left( \theta _{1}\right) $ we get
\begin{equation}
\eta\left( x_{n+1},x_{n}\right) <\eta\left( x_{n},x_{n-1}\right).
\end{equation}
Therefore, $d\left( x_{n+1,}x_{n}\right) _{n\in \mathbb{N}}$ is monotone strictly decreasing sequence of non negative real numbers. Consequently, there exists $\lambda \geq 0$:
\begin{equation*}
\lim_{n\rightarrow \infty }\eta\left( x_{n+1,}x_{n}\right) \ \ =\lambda.\\
\end{equation*}
Now, we claim that $\lambda =0$. Arguing by contradiction, we assume that $\alpha >0.$ Since $\eta\left( x_{n+1,}x_{n}\right) _{n\in \mathbb{N}}$ is a non negative decreasing sequence, then we have
\begin{equation*}
\eta\left( x_{n+1,}x_{n}\right)  \geq \alpha \text{ \ \ }\forall n\in \mathbb{N}.
\end{equation*}
By property of $ \theta $ we get,
\begin{equation}\label{qq}
1<\theta \left(\lambda\right) \leq \theta \left(\eta\left(x_{1},x_{0}\right)\right) ^{r^{n}}.
\end{equation}
By letting $n\rightarrow \infty $ in inequality $\left( \ref{qq}\right),$ we obtain
\begin{equation*}
1<\theta \left(\lambda \right) \leq 1.
\end{equation*}
It is a contradiction. Therefore,
\begin{equation}\label{3.6}
\lim_{n\rightarrow \infty }\eta\left( x_{n+1,}x_{n}\right) =0.
\end{equation}
Substituting $x=x_{n+1}$ and $y=x_{n-1}$, from $(\ref{3.1})$, for all $n\in  $ $\mathbb{N}$, we have
\begin{equation}\label{3.2}
\theta \left[ \eta\left( x_{n+2},x_{n}\right)\right]\leq \theta \left[ s^{2}\eta\left( x_{n+2},x_{n}\right)\right] \leq \left[ \theta \left( \eta\left(x_{n},x_{n-1}\right) \right) \right] ^{r},\forall n\in \mathbb{N}.
\end{equation}
Repeating this step, we conclude that
\begin{equation*}\theta \left( \eta\left( x_{n+2},x_{n}\right) \right)  \leq \left(\theta \left( \eta\left( x_{n+1},x_{n-1}\right) \right) \right) ^{r} \\
\leq \left(\theta \left( \eta\left( x_{n},x_{n-2}\right) \right) \right) ^{r^{2}}\leq ...\leq\theta\left( \eta\left( x_{2},x_{0}\right)\right) ^{r^{n}}.
 \end{equation*}
By $(3.36) $ and using $\left( \theta _{1}\right) $ we get
\begin{equation}
\eta\left( x_{n+2},x_{n}\right) <\eta\left( x_{n+1},x_{n-1}\right).
\end{equation}
Therefore, $d\left( x_{n+2,}x_{n}\right) _{n\in \mathbb{N}}$ is monotone strictly decreasing sequence of non negative real numbers. Consequently, there exists $\beta \geq 0$:
\begin{equation*}
\lim_{n\rightarrow \infty }\eta\left( x_{n+2,}x_{n}\right) \ \ =\beta.
\end{equation*}
Now, we claim that $\beta =0$. Arguing by contradiction, we assume that $\delta >0.$ Since $\eta\left( x_{n+2,}x_{n}\right) _{n\in \mathbb{N}}$ is a non negative decreasing sequence, then we have
\begin{equation*}
\eta\left( x_{n+2,}x_{n}\right)  \geq \beta \text{ \ \ }\forall n\in \mathbb{N}.
\end{equation*}
By property of $ \theta $ we get,
\begin{equation}\label{xx}
1<\theta \left( \beta\right) \leq \theta \left(\eta\left(x_{2},x_{0}\right)\right) ^{r^{n}}.
\end{equation}
By letting $n\rightarrow \infty $ in inequality $\left( \ref{xx}\right),$ we obtain
\begin{equation*}
1<\theta \left(\beta\right) \leq 1.
\end{equation*}
It is a contradiction. Therefore,
\begin{equation}\label{3.6}
\lim_{n\rightarrow \infty }\eta\left( x_{n+2},x_{n}\right) =0.
\end{equation}
Firstly we show  $\left\lbrace  x_{n}\right\rbrace  _{n\in \mathbb{N}}$ is a right-CS, if otherwise there exists an $\varepsilon $ $>0$ for which we can find sequences of positive integers $\left\lbrace n_{\left( k\right) }\right\rbrace $ and $\left\lbrace m_{\left( k\right) }\right\rbrace $ such that, for all positive integers
$  k, n_{k} > m_{k} > k$. By Lemma $ (\ref{4.4}) $, we have
$$\varepsilon \leq \lim_{k\rightarrow \infty }\inf \eta\left( x_{m_{\left( k\right) }},x_{n_{\left( k\right)}}\right)  \leq \lim_{k\rightarrow \infty }\sup \eta\left( x_{m_{\left( k\right) }},x_{n_{\left( k\right)}}\right)\leq s\varepsilon ,$$
  $$\varepsilon \leq \lim_{k\rightarrow \infty }\inf \eta\left( x_{n_{\left( k\right) }},x_{m_{\left( k\right)+1}}\right)  \leq \lim_{k\rightarrow \infty }\sup \eta\left( x_{n_{\left( k\right) }},x_{m_{\left( k\right)+1}}\right)\leq s\varepsilon ,$$
  $$\varepsilon \leq \lim_{k\rightarrow \infty }\inf \eta\left( x_{m_{\left( k\right) }},x_{n_{\left( k\right)+1}}\right)  \leq \lim_{k\rightarrow \infty }\sup \eta\left( x_{m_{\left( k\right) }},x_{n_{\left( k\right)+1}}\right)\leq s\varepsilon ,$$
  $$\frac{\varepsilon}{s} \leq \lim_{k\rightarrow \infty }\inf \eta\left( x_{m_{\left( k\right)+1 }},x_{n_{\left( k\right)+1}}\right)  \leq \lim_{k\rightarrow \infty }\sup \eta\left( x_{m_{\left( k\right)+1 }},x_{n_{\left( k\right)+1}}\right)\leq s^{2}\varepsilon $$
Applying $ (\ref{3.1}) $ with $ x=x_{m_{\left( k\right) }} $ and  $ y=x_{n_{\left( k\right) }}$, we obtain
\begin{equation}
\theta \left[ s^{2}\eta\left( x_{m_{\left( k\right) +1}},x_{n_{\left( k\right)+1}}\right)\right]  \leq \left[ \theta \left( M\left( x_{m_{\left( k\right)}},x_{n_{\left( k\right) }}\right) \right)\right]  ^{r}.
\end{equation}
Letting $ k\rightarrow \infty $ the above inequality, applying the continuity of $ \theta $, we obtain
\begin{equation*}
\theta\left( \frac{\varepsilon}{s}s^{2}\right) =\theta\left( \varepsilon s\right) \leq \theta \left(s^{2}\lim_{k\rightarrow \infty }\eta\left( x_{m_{\left( k\right)+1}},x_{n_{\left( k\right) +1}}\right)\right) \leq  \left[ \theta \left ( \lim_{k\rightarrow \infty }\eta\left( x_{m_{\left( k\right)}},x_{n_{\left( k\right) }}\right) \right)\right]  ^{r}.
\end{equation*}
Therefore, $$ \theta(s\varepsilon)\leq \left[ \theta(s\varepsilon)\right]^{r} <\theta(s\varepsilon). $$
Since $ \theta $ is increasing, and continuous function, we get
\begin{equation*}
 s \varepsilon  <s \varepsilon,
\end{equation*}
which is a contradiction. Then
\begin{equation*}
\lim_{n,m\rightarrow \infty }\eta\left( x_{m},x_{n}\right) =0.
\end{equation*}
Consequently $\left\lbrace  x_{n}\right\rbrace  $ is a right-Cauchy in $\mathcal{X}$. By completeness of $\left( \mathcal{X},\eta\right)$, there exists $z\in \mathcal{X}$:
\begin{equation*}
\lim_{n\rightarrow \infty }\eta\left( x_{n},z\right)=0.
\end{equation*}
Secondly we show $\left\lbrace  x_{n}\right\rbrace  _{n\in \mathbb{N}}$ is a left-CS, if otherwise there exists an $\varepsilon $ $>0$ for which we can find sequences of positive integers $\left\lbrace m_{\left( k\right) }\right\rbrace $ and $\left\lbrace n_{\left( k\right) }\right\rbrace $ such that, for all positive integers
$  k, m_{k} > n_{k} > k$. By Lemma $ (\ref{4.4}) $, we have
 $$\varepsilon \leq \lim_{k\rightarrow \infty }\inf \eta\left( x_{n_{\left( k\right) }},x_{m_{\left( k\right)}}\right)  \leq \lim_{k\rightarrow \infty }\sup \eta\left( x_{n_{\left( k\right) }},x_{m_{\left( k\right)}}\right)\leq s\varepsilon ,$$
  $$\varepsilon \leq \lim_{k\rightarrow \infty }\inf \eta\left( x_{m_{\left( k\right) }},x_{n_{\left( k\right)+1}}\right)  \leq \lim_{k\rightarrow \infty }\sup \eta\left( x_{m_{\left( k\right) }},x_{n_{\left( k\right)+1}}\right)\leq s\varepsilon ,$$
  $$\varepsilon \leq \lim_{k\rightarrow \infty }\inf \eta\left( x_{n_{\left( k\right) }},x_{m_{\left( k\right)+1}}\right)  \leq \lim_{k\rightarrow \infty }\sup \eta\left( x_{n_{\left( k\right) }},x_{m_{\left( k\right)+1}}\right)\leq s\varepsilon ,$$
  $$\frac{\varepsilon}{s} \leq \lim_{k\rightarrow \infty }\inf \eta\left( x_{n_{\left( k\right)+1 }},x_{m_{\left( k\right)+1}}\right)  \leq \lim_{k\rightarrow \infty }\sup \eta\left( x_{n_{\left( k\right)+1 }},x_{m_{\left( k\right)+1}}\right)\leq s^{2}\varepsilon .$$
Applying $ (\ref{3.1}) $ with $ x=x_{n_{\left( k\right) }} $ and  $ y=x_{m_{\left( k\right) }}$, we obtain
\begin{equation}\label{ccc}
\theta \left[ s^{2}\eta\left( x_{n_{\left( k\right) +1}},x_{m_{\left( k\right)+1}}\right)\right]  \leq \left[ \theta \left( \eta\left( x_{n_{\left( k\right)}},x_{m_{\left( k\right) }}\right) \right)\right]  ^{r}.
\end{equation}
Letting $ k\rightarrow \infty $ the above inequality, applying the continuity of $ \theta $ and using  $(\ref{ccc} )$, we obtain
\begin{equation*}
\theta\left( \frac{\varepsilon}{s}s^{2}\right) =\theta\left( \varepsilon s\right) \leq \theta \left(s^{2}\lim_{k\rightarrow \infty }\eta\left( x_{n_{\left( k\right)+1}},x_{m_{\left( k\right) +1}}\right)\right) \leq  \left[ \theta \left ( \lim_{k\rightarrow \infty }\eta\left( x_{n_{\left( k\right)}},x_{m_{\left( k\right) }}\right) \right)\right] ^{r}.
\end{equation*}
Therefore, $$ \theta(s\varepsilon)\leq \left[ \theta(s\varepsilon)\right]^{r} <\theta(s\varepsilon). $$
Since $ \theta $ is increasing, we get
\begin{equation*}
 s \varepsilon  <s \varepsilon,
\end{equation*}
which is a contradiction. Then
\begin{equation*}
\lim_{n,m\rightarrow \infty }\eta\left( x_{n},x_{m}\right) =0.
\end{equation*}
Consequently $\left\lbrace  x_{n}\right\rbrace  $ is a left-Cauchy in $\mathcal{X}$. By completeness of $\left( \mathcal{X},\eta\right) ,$ there exists $u\in \mathcal{X}$:
\begin{equation*}
\lim_{n\rightarrow \infty }\eta\left( u,x_{n}\right) =0.
\end{equation*}
By the lemma $  \ref{lemmkari}$, we conclude that $ z=u $.
Now, we show that $\eta\left( \mathcal{T}z,z\right) =0$ or $\eta\left( z,\mathcal{T}z\right) =0$ arguing by contradiction, we assume that
\begin{equation*}
 \eta\left( \mathcal{T}z,z\right)>0.
\end{equation*}
Since $ x_{n} $ forward converges to $x$ as $ n\rightarrow \infty $ for all $ n\in \mathbf{N} $, then from Lemma $ (\ref{2.3}) $, we conclude that
\begin{equation}\label{12}
\frac{1}{s}\eta\left( z,\mathcal{T}z\right)\leq \lim_{n\rightarrow \infty }\sup \eta\left(\mathcal{T}x_{n},\mathcal{T}z\right) \leq s\eta\left( z,\mathcal{T}z\right).
\end{equation}
Now, applying $\left(\ref{3.1}\right) $ with $ x=x_n $ and $ y=z $, we have
\begin{equation}\label{22}
\theta \left( s^{2}\eta\left( \mathcal{T}x_{n},\mathcal{T}z\right)\right) \leq \left[ \theta\left( \eta\left( x_{n},z\right) \right) \right] ^{r},\text{ }\forall n\in \mathbb{N}.
\end{equation}
By letting $n\rightarrow \infty $ in inequality $(\ref{22}) $,  using $(\ref{12} ) $ and $\theta_3 $ we obtain
\begin{align*}
\theta\left[ s^{2}\frac{1}{s}\eta\left( z,\mathcal{T}z\right)\right]& =\theta\left[s \eta\left( z,\mathcal{T}z\right)\right] \\
&\leq \theta\left[ s^{2}\lim_{n\rightarrow  \infty } \eta\left(\mathcal{T}x_{n},\mathcal{T}z\right)\right] \\
& \leq \left[\theta \left( \eta\left( z,\mathcal{T}z\right)\right) \right]^{r}\\
&<\theta \left( \eta\left( z,\mathcal{T}z\right)\right).
\end{align*}
By $\left( \theta _{1}\right)$ and $\left( \theta _{3}\right)$, we get
$$ s\eta (z, \mathcal{T}z) < \eta (z, \mathcal{T}z),$$
therefore
$$ \eta (z,\mathcal{T}z) (s - 1) < 0 \Rightarrow  s < 1.$$
Which is a contradiction. Hence  $  \eta\left( \mathcal{T}z,z\right)=0$.
Next applying $\left(\ref{3.1}\right) $ with $ x=z $ and $ y=x_n $, we have
\begin{equation}\label{134}
\theta \left( s^{2}\eta\left( \mathcal{T}z,\mathcal{T}x_{n}\right)\right) \leq \left[ \theta\left( \eta\left( z,x_{n}\right) \right) \right] ^{r},\text{ }\forall n\in \mathbb{N},
\end{equation}
Since $ x_{n} $ backward converges to $x$ as $ n\rightarrow \infty $ for all $ n\in \mathbf{N} $, then from Lemma $ (\ref{2.3}) $, we conclude that
\begin{equation}\label{13}
\frac{1}{s}d\left(\mathcal{T} z,z\right)\leq \lim_{n\rightarrow \infty }\sup \eta\left(\mathcal{T}z,\mathcal{T}x_{n}\right) \leq s\eta\left( \mathcal{T}z,z\right).
\end{equation}
By letting $n\rightarrow \infty $ in inequality $(\ref{13}) $, using $(\ref{134} ) $ and $\theta_3 $ we obtain
\begin{align*}
\theta\left[ s^{2}\frac{1}{s}\eta\left( \mathcal{T}z,z\right)\right]& =\theta\left[s \eta\left( \mathcal{T}z,z\right)\right] \\
&\leq \theta\left[ s^{2}\lim_{n\rightarrow  \infty } \eta\left(\mathcal{T}z,\mathcal{T}x_{n}\right)\right] \\
& \leq \left[\theta \left( \eta\left( \mathcal{T}z,z\right)\right) \right]^{r}\\
&<\theta \left( \eta\left( \mathcal{T}z,z\right)\right).
\end{align*}
By $\left( \theta _{1}\right)$ and $\left( \theta _{3}\right)$, we get
$$ s\eta (\mathcal{T}z, z) < \eta (\mathcal{T}z, z),$$
therefore
$$ \eta (\mathcal{T}z,z) (s - 1) < 0 \Rightarrow  s < 1.$$
Which is a contradiction. Hence  $  \eta\left( \mathcal{T}z,z\right)=0$.\\
To prove the uniqueness. Now, suppose that $z,u\in \mathcal{X}$ are two FPs of $\mathcal{T}$ such that $u\neq z$. Therefore, we have
\begin{equation*}
\eta\left( z,u\right) =\eta\left( \mathcal{T}z,\mathcal{T}u\right) >0.
\end{equation*}
Applying $\left( \ref{3.1}\right) $ with $ x=z $ and $ y=u $, we have
\begin{equation*}
\theta \left( \eta\left( z,u\right)\right)=\theta \left( \eta\left( \mathcal{T}u,\mathcal{T}z\right)\right)\leq \theta \left(s^{2} \eta\left( \mathcal{T}u,\mathcal{T}z\right)\right) \leq \left[ \theta \left( \eta\left( z,u\right)\right) \right] ^{r}
\end{equation*}
\begin{equation*}
\theta \left( \eta\left( z,u\right)\right)\leq\left[ \theta\left(\eta\left( z,u\right)\right) \right] ^{r} \\
<\theta \left(\eta\left( z,u\right)\right) \\
\end{equation*}
which implies that
\begin{equation*}
\eta\left( z,u\right) <\eta\left( z,u\right),
\end{equation*}
which is a contradiction. Therefore $u=z$.
\end{proof}
\begin{corollary}
Let $\left( \mathcal{X},\eta\right) $ be a complete QRB-MS and $\mathcal{T}:\mathcal{X}\rightarrow \mathcal{X}$ $\ $be given mapping. Suppose that there exist $\theta\in \Theta$ and $k\in \left] 0,1\right[ $ such that for any $x,y\in \mathcal{X},$\\ we have
\end{corollary}
 $$\eta\left( \mathcal{T}x,\mathcal{T}y\right) >0\Rightarrow \left[ s^{2}\eta\left( \mathcal{T}x,\mathcal{T}y\right)\right] \leq k\left[ \left( \eta\left( x,y\right)\right) \right] .$$
Then $\mathcal{T}$ has a unique FP.
\begin{example}
	Let $ \mathcal{X}=\left[1,2 \right]  $.\\ Define $ \eta:\mathcal{X}\times \mathcal{X}\rightarrow \left[0,+\infty \right[  $ by
	\begin{equation*}
	\eta(x,y)=\left\lbrace
	\begin{aligned}
	 (x-y) ^{2}	& \ if \ x\geq y\\
	\frac{1}{2} (y-x) ^{2}&  \ if \ x<y.
	\end{aligned}
	\right.
	\end{equation*}
Then $ (\mathcal{X},\eta) $ is a QRB-MS with coefficient s=2. Define mapping  $\mathcal{T}:\mathcal{X}\rightarrow \mathcal{X}$ by
	\begin{equation*}
	\mathcal{T}(x)=\sqrt{x}.
	\end{equation*}
Evidently, $\mathcal{T}(x)\in \mathcal{X} $. Let $\theta \left( t\right) =e^{\sqrt{t}},$ $r =\frac{1}{2}$. It obvious that $\theta \in \Theta $ and $r \in \left]0,1 \right[.$
Consider the following possibilities:\\
	\item[1] : $ x \geq y $. Then
	 $$ \mathcal{T}(x)=x^{\frac{1}{4}},\   \mathcal{T}(y)=y^{\frac{1}{4}},\ \eta\left( \mathcal{T}x,\mathcal{T}y \right)=(x^{\frac{1}{4}}-y^{\frac{1}{4}})^{2}.  $$
	 On the other hand
	 $$ \theta \left[s^{2}\eta\left(\mathcal{T}x,\mathcal{T}y\right) \right]=e^{2(x^{\frac{1}{4}}-y^{\frac{1}{4}})}     $$
	 and
\begin{align*}
	 \eta(x,y)&=(x-y)^{2}.
\end{align*}	
On the other hand
	$$\left[ \theta \left(\eta\left(x,y\right) \right) \right]^{r}= e^{\frac{1}{2}(x-y)} .$$
Since $$ \frac{1}{2}(x-y)=\frac{1}{2}(x^\frac{1}{4}-y^\frac{1}{4})(x^\frac{1}{4}+y^\frac{1}{4})(x^\frac{1}{4}+y^\frac{1}{4})(x^\frac{1}{4}+y^\frac{1}{4}                                            ) $$
Since $ x\geq 1 $, so
 $$2(x^{\frac{1}{4}}-y^{\frac{1}{4}})\leq  \frac{1}{2}(x^\frac{1}{4}-y^\frac{1}{4})(x^\frac{1}{4}+y^\frac{1}{4})(x^\frac{1}{4}+y^\frac{1}{4})(x^\frac{1}{4}+y^\frac{1}{4} ),     $$	
therefore
	\begin{align*}
	\theta(s^{2} \eta(\mathcal{T}x,\mathcal{T}y) &\leq  \left[ \theta( \eta(x,\mathcal{T}x))\right]^{r}.
	\end{align*}
	\item[2]: $ x<y $. Then
	 $$ \mathcal{T}(x)=x^{\frac{1}{4}},\ \mathcal{T}(y)=y^{\frac{1}{4}},\ \eta\left( \mathcal{T}x,\mathcal{T}y \right)=\frac{1}{2}(x^{\frac{1}{4}}-y^{\frac{1}{4}})^{2}.  $$
On the other hand
	 $$ \theta \left[s^{2}\eta\left(\mathcal{T}x,\mathcal{T}y\right) \right]=e^{\sqrt{2}(x^{\frac{1}{4}}-y^{\frac{1}{4}})},$$
and
\begin{align*}
	 \eta(x,y)&=\frac{1}{2}(x-y)^{2}.
\end{align*}
On the other hand
	$$\left[ \theta \left(\eta\left(x,y\right) \right) \right]^{r}= e^{\frac{1}{2\sqrt{2}}(x-y)} .$$
Since $$ \frac{1}{2}(x-y)=\frac{1}{2}(x^\frac{1}{4}-y^\frac{1}{4})(x^\frac{1}{4}+y^\frac{1}{4})(x^\frac{1}{4}+y^\frac{1}{4})(x^\frac{1}{4}+y^\frac{1}{4}                                            ) $$
Since $ x\geq 1 $, so
 $$2(x^{\frac{1}{4}}-y^{\frac{1}{4}})\leq  \frac{1}{2}(x^\frac{1}{4}-y^\frac{1}{4})(x^\frac{1}{4}+y^\frac{1}{4})(x^\frac{1}{4}+y^\frac{1}{4})(x^\frac{1}{4}+y^\frac{1}{4} ),     $$	
therefore
	\begin{align*}
	\theta(s^{2} \eta(\mathcal{T}x,\mathcal{T}y) &\leq  \left[ \theta( \eta(x,\mathcal{T}x))\right]^{r}.
	\end{align*}
Hence, the condition $ (\ref{3.1}) $ is satisfied. Therefore, $ \mathcal{T} $ has a unique FP $ z=\frac{1}{3} $.
\end{example}

\begin{theorem}\label{4}
Let $\left( \mathcal{X},\eta\right) $ be a complete b-MS and let $\mathcal{T}:\mathcal{X}\rightarrow \mathcal{X}$ be an $5 \theta-\phi $-contraction, i.e, there exist $\theta \in \Theta $ and $\phi\in \Phi $ such that for any $ x,y \in \mathcal{X} $, we have
\begin{equation}\label{4.A}
\eta\left( \mathcal{T}x,\mathcal{T}y\right) >0\Rightarrow \theta \left[s^{2}\eta\left( \mathcal{T}x,\mathcal{T}y\right)\right]  \leq  \phi\left[ \theta \left( \eta\left( x,y\right) \right) \right].
\end{equation}
Then $\mathcal{T}$ has a unique FP.
\end{theorem}
\begin{proof}
The proof can easily sketched following the same scenario as in the proof of Theorem \ref{Theorem3.5}.
\end{proof}

\begin{corollary}
Let $\left( \mathcal{X},\eta\right) $ be a complete QRB-MS and $\mathcal{T}:\mathcal{X}\rightarrow \mathcal{X}$ $\ $be given mapping. Suppose that there exist $\theta\in \Theta$ and $k\in \left] 0,1\right[ $ such that for any $a_1,a_2\in \mathcal{X},$\\ we have
\end{corollary}
 $$\eta\left( \mathcal{T}a_1,\mathcal{T}a_2\right) >0\Rightarrow \theta\left[ s^{2}\eta\left( \mathcal{T}a_1,\mathcal{T}a_2\right)\right] \leq \theta\left[ \left( \eta\left( a_1,a_2\right)\right) \right]^{k} .$$
Then $\mathcal{T}$ has a unique FP.

\begin{example}
	Let $ \mathcal{X}=A\cup B $, where $ A=\lbrace \frac{1}{n}:n\in\lbrace 3,4,5,6\rbrace \rbrace $ and $ B=\left[\frac{1}{2},\frac{3}{2} \right]  $. Define $ \eta:\mathcal{X}\times \mathcal{X}\rightarrow \left[0,+\infty \right[  $ as follows:
	\begin{equation*}
	\left\lbrace
	\begin{aligned}
	\eta(a_1, a_2) &=\eta(a_2, a_1)\ for \ all \  a_1,a_2\in \mathcal{X};\\
	\eta(a_1, a_2) &=0\Leftrightarrow a_2= a_1.\\	
	\end{aligned}
	\right.
	\end{equation*}
	and
	\begin{equation*}
	\left\lbrace
	\begin{aligned}		
	\eta\left( \frac{1}{3},\frac{1}{4}\right) =\eta\left( \frac{1}{4},\frac{1}{5}\right) &=0,1\\
	\eta\left( \frac{1}{4},\frac{1}{3}\right) =\eta\left( \frac{1}{5},\frac{1}{4}\right) &=0,05\\
	\eta\left( \frac{1}{3},\frac{1}{5}\right) =\eta\left( \frac{1}{4},\frac{1}{6}\right) &=0,05\\
	\eta\left( \frac{1}{5},\frac{1}{3}\right) =\eta\left( \frac{1}{6},\frac{1}{4}\right) &=0,1\\
	\eta\left( \frac{1}{3},\frac{1}{6}\right) =\eta\left( \frac{1}{5},\frac{1}{6}\right)&=0,5\\
	\eta\left( a_1,a_2\right) =\left( \vert a_1-a_2\vert\right) ^{2} \ otherwise.
	\end{aligned}
	\right.
	\end{equation*}
	Then $ (\mathcal{X},\eta) $ is a b-rectangular MS with coefficient s=3. However we have the following:
	\item[1)] $ (\mathcal{X},\eta) $ is not a MS, as $\eta\left( \frac{1}{5},\frac{1}{6}\right)=0.5>0.15=\eta\left( \frac{1}{5},\frac{1}{4}\right)+\eta\left( \frac{1}{4},\frac{1}{6}\right) 	$.
	\item[2)] $ (\mathcal{X},\eta) $ is not a  b-MS for s=3, as $\eta\left( \frac{1}{5},\frac{1}{6}\right)=0.5>0.45=3\left[ \eta\left( \frac{1}{5},\frac{1}{4}\right)+\eta\left( \frac{1}{4},\frac{1}{6}\right)\right]  	$.
	\item[3)] $ (\mathcal{X},\eta) $ is not a rectangular MS, as $\eta\left( \frac{1}{5},\frac{1}{6}\right)=0.5>0.2=\eta\left( \frac{1}{5},\frac{1}{3}\right)+\eta\left( \frac{1}{3},\frac{1}{4}\right)+\eta\left( \frac{1}{4},\frac{1}{6}\right)$.
	\item[4)] $ (\mathcal{X},\eta) $ is not a b-rectangular MS, as $  \eta\left( \frac{1}{3},\frac{1}{4}\right)\neq \eta\left( \frac{1}{4},\frac{1}{3}\right)$.
	Define mapping $\mathcal{T}:\mathcal{X}\rightarrow \mathcal{X}$ by
	\begin{equation*}
	\mathcal{T}(a_1)=\left\lbrace
	\begin{aligned}
	\frac{\sqrt{a_1}+3}{4}	& \ if \ a_1\in \left[\frac{1}{2},\frac{3}{2} \right]\\
	1 &  \ if \ a_1\in A.\\
	\end{aligned}
	\right.
	\end{equation*}
	Then, $\mathcal{T}(a_1)\in \left[\frac{1}{2},\frac{3}{2} \right]$. Let $\theta \left( t\right) =\sqrt{t}+1,$ $\phi \left( t\right) =\frac{t+1}{2}$. It obvious that $\theta \in \Theta $ and $\phi \in\Phi .$\\
	Consider the following possibilities:\\
	case 1: $ a_1,a_2 \in \left[\frac{1}{2},\frac{3}{2} \right]$ with $ a_1\neq a_2 $, assume that $ a_1>a_2 $.
	$$\eta(\mathcal{T}a_1,\mathcal{T}a_2)=\left( \frac{\sqrt{a_1}-\sqrt{a_2}}{4}\right) ^ {2}  .$$
	Therefore,
	$$ \theta(s^{2} \eta(\mathcal{T}a_1,\mathcal{T}a_2)= \frac{3}{4}\left( \sqrt{a_1}-\sqrt{a_2}\right) +1$$
	and
	$$ \phi\left[ \theta( \eta(a_1,a_2))\right] =\frac{a_1-a_2}{2}+1.$$
	On the other hand
	\begin{align*}
	\theta(s^{2} \eta(\mathcal{T}a_1,\mathcal{T}a_2)- \phi\left[ \theta( \eta(a_1,a_2))\right]& =\frac{3\left( \sqrt{a_1}-\sqrt{a_2}\right)-2( a_1-a_2)}{4}\\
	&=\frac{1}{4}\left(\left( \sqrt{a_1}-\sqrt{a_2}\right) \right)\left[  3-2\left( \sqrt{a_1}+ \sqrt{a_2}\right) \right].
	\end{align*}
	Since $ a_1,a_2\in \left[\frac{1}{2},\frac{3}{2} \right]$, then
	$$\frac{1}{4}\left(\left( \sqrt{a_1}-\sqrt{a_2}\right) \right)\left[  3-2\left( \sqrt{a_1}+ \sqrt{a_2}\right) \right]\leq 0  .$$
Hence
	\begin{align*}
	\theta(s^{2} \eta(\mathcal{T}a_1,\mathcal{T}a_2) &\leq  \phi\left[ \theta( \eta(a_1,a_2))\right].
	\end{align*}
	case 2: $ a_1\in \left[\frac{1}{2},\frac{3}{2} \right], a_2 \in A  $ or $ a_2\in \left[\frac{1}{2},\frac{3}{2} \right],a_1 \in A  $ . \\
Therefore, $ \mathcal{T}(a_1)=\frac{\sqrt{a_1}+3}{4} $, $ \mathcal{T}(a_2)=1 $, then $ \eta(\mathcal{T}a_1,\mathcal{T}a_2)=\left( \vert \frac{\sqrt{a_1}-1}{4}   \vert \right) ^{2} $\\
In this case consider two possibilities:\\	
$ i) $ $a_1\geq 1:$
 Then $\sqrt{a_1}\geq 1.$ Therefore,
$$ \eta(\mathcal{T}a_1,\mathcal{T}a_2) = \left( \frac{\sqrt{a_1}-1}{4}\right)^{2} . $$
So, we have
	$$ \theta(s^{2} \eta(\mathcal{T}a_1,\mathcal{T}a_2)= \frac{3}{4}\left( \sqrt{a_1}-1\right) +1.$$
Hence
	\begin{align*}
	\theta(s^{2} \eta(\mathcal{T}a_1,\mathcal{T}a_2) &\leq  \phi\left[ \theta( \eta(a_1,a_2))\right].
	\end{align*}
$ ii) $ $ a_1<1: $
 Then $\sqrt{a_1}<1.$ Therefore,
$$ \eta(\mathcal{T}a_1,\mathcal{T}a_2)=\left( \vert \frac{1-\sqrt{a_1}}{4}   \vert \right)^{2}  = \left( \frac{1-\sqrt{a_1}}{4}\right) ^{2} . $$
	Since, $ a_1\in \left[\frac{1}{2},1 \right], $ then
	\begin{align*}
	\theta(s^{2} \eta(\mathcal{T}a_1,\mathcal{T}a_2) &\leq  \phi\left[ \theta( \eta(a_1,a_2))\right].
	\end{align*}
	Hence, the condition $ (\ref{4.A}) $ is satisfied. Therefore, $ \mathcal{T} $ has a unique FP $ z=1 $.
\end{example}
\medskip

\section*{Declarations}

\medskip

\noindent \textbf{Availability of data and materials}\newline
\noindent Not applicable.

\medskip

\noindent \textbf{Human and animal rights}\newline
\noindent We would like to mention that this article does not contain any studies with animals and does not
involve any studies over human being.

\medskip

\noindent \textbf{Competing  interest}\newline
\noindent On behalf of all authors, the corresponding author states that there is no conflict of interest.

\medskip

\noindent \textbf{Funding}\newline
\noindent  Authors declare that there is no funding available for this article.

\medskip

\noindent \textbf{Authors' contributions}\newline
\noindent The authors equally conceived of the study, participated in its
design and coordination, drafted the manuscript, participated in the
sequence alignment, and read and approved the final manuscript.

\end{document}